\documentclass[12pt]{amsart}
\usepackage{latexsym, amsmath, amssymb,amsthm,amsopn,amsfonts}
\usepackage{version}
\usepackage{epsfig,graphics,color,graphicx,graphpap}
\usepackage{amssymb}

\ifx\pdfoutput\undefined
  \DeclareGraphicsExtensions{.eps}
\else
  \ifx\pdfoutput\relax
    \DeclareGraphicsExtensions{.eps}
  \else
    \ifnum\pdfoutput>0
      \DeclareGraphicsExtensions{.pdf}
    \else
      \DeclareGraphicsExtensions{.eps}
    \fi
  \fi
\fi

\newcommand{\alp}{\alpha}
\newcommand{\eps}{\varepsilon}
\newcommand{\Dcal}{\mathcal D}

\newcommand{\Zcal}{\mathcal Z}

\newcommand{\CI}{{\mathcal C}^\infty }

\newcommand{\D}{{\mathcal D}}
\newcommand{\ZZ}{{\mathbb Z}}

\newcommand{\RR}{{\mathbb R}}

\newcommand{\NN}{{\mathbb N}}

\newcommand{\sinL}{\sin\left(k\pi \frac{y}{L(x)}\right)}
\newcommand{\cosL}{\cos\left(k\pi \frac{y}{L(x)}\right)}

\theoremstyle{plain}

\newtheorem{thm}{Theorem}
\newtheorem{prop}{Proposition}[section]
\newtheorem{cor}[prop]{Corollary}
\newtheorem{lem}[prop]{Lemma}

\theoremstyle{definition}

\newtheorem{rem}{Remark}[section]

\numberwithin{equation}{section}

\def\squarebox#1{\hbox to #1{\hfill\vbox to #1{\vfill}}}

\newcommand{\p}{\partial}
\usepackage{amsxtra}

\newcommand{\und}{\frac{1}{2}}

\newcommand{\refeq}[1]{(\ref{#1})}

\title
[Partially Rectangular Billiards] 
{Nonconcentration in partially rectangular billiards}

\author[L. Hillairet]
{Luc Hillairet}
\email{Luc.Hillairet@math.univ-nantes.fr}
\address{UMR CNRS 6629-Universit\'{e} de Nantes, 2 rue de la Houssini\`{e}re, \\
BP 92 208, F-44 322 Nantes Cedex 3, France}

\author[J.L. Marzuola]
{Jeremy L. Marzuola}
\email{marzuola@math.unc.edu}

\address{Mathematics Department, University of North Carolina, Chapel Hill \\
Phillips Hall, Chapel Hill, NC 27798, USA}

\begin{document}    

\begin{abstract}
In specific types of partially rectangular billiards we estimate 
the mass of an eigenfunction of energy $E$ in the region outside the rectangular set 
in the high-energy limit. We use the adiabatic ansatz to compare 
the Dirichlet energy form with a second quadratic form for which separation of variables 
applies. This allows us to use sharp one-dimensional control estimates and to derive 
the bound assuming that $E$ is not resonating with the Dirichlet spectrum of 
the rectangular part.  
\end{abstract}
   
\maketitle 

\section{Introduction}
We study concentration and non-concentration of eigenfunctions of the Laplace 
operator in stadium-like billiards. As predicted by the quantum/classical correspondence, 
such concentration is deeply linked with the classical underlying dynamics. In particular, the 
celebrated quantum ergodicity theorem roughly states that when the corresponding 
classical dynamics is ergodic then almost every sequence of eigenfunctions equidistributes 
in the high energy limit (see \cite{Shn, CdV, Zel} and \cite{GerLei, ZZ} in the billiard 
setting for a more precise statement). In strongly chaotic systems such as negatively curved manifolds, 
it is expected that every sequence of eigenfunctions equidistributes. This statement is 
the Quantum Unique Ergodicity conjecture (Q.U.E.) and remains open in most cases despite several 
recent striking results (see for instance \cite{BFN, Lin, Ana, Ana-Non}). 
On the other extreme, the Bunimovich stadium, although ergodic, is expected to 
violate Q.U.E.. Indeed, it is expected that there exist bouncing ball modes i.e. 
exceptional sequences of eigenfunctions concentrating on the cylinder of bouncing 
ball periodic orbits that sweep out the rectangular region (see \cite{BSS} for instance). 
The existence of such bouncing ball modes is still open and only recently 
did Hassell prove that the generic Bunimovich stadium billiard indeed fails 
to be Q.U.E. (see \cite{Hass}).       

Our work is closely related to the search for bouncing ball modes but 
proceeds loosely speaking in the other direction. We actually aim at 
understanding how strong concentration of eigenfunctions 
in the rectangular part cannot be. We thus follow \cite{BZ3} 
in which Burq-Zworski proved that even bouncing ball modes couldn't concentrate 
strictly inside the rectangular region. This was made precise by 
Burq-Hassell-Wunsch in \cite{BHW} where the following estimate was proved :  
\[ 
\|u\|_{L^2(W)} \geq E^{-1}\|u\|_{L^2(\Omega)}
\]
in which $\|u\|_{L^2(W)}$ (resp. $\|u\|_{L^2(\Omega)}$) denotes the $L^2$ norm of the eigenfunction $u$ in the 
wings (resp. in the billiard).

Our main result for the Bunimovich stadium is the following theorem 

\begin{thm}\label{thm:intro}
Let $\Omega$ be a Bunimovich stadium with rectangular part 
$R:= [-B_0,0]\times [0,L_0].$ We set $W=\Omega\setminus R$ and 
denote by $\Sigma$ the Dirichlet spectrum of $R,$ i.e. 
\[
\Sigma\,=\,\left\{ \frac{k^2\pi^2}{L_0^2}\,+\, \frac{l^2\pi^2}{B_0^2},~k,l\in \NN \right \}.
\] 
For any $\eps\geq 0$ there exists $E_0$ and $C$ such that 
if $u$ is an eigenfunction of energy $E$ such that $E>E_0$ and 
$\mathrm{dist}(E, \Sigma)> E^{-\eps}$ then the following estimates holds:
\[ 
\|u\|_{L^2(\Omega)}\,\leq C E^{\frac{5+8\eps}{6}} \| u\|_{L^2(W)},
\]
\end{thm}

This bound improves on the Burq-Hassell-Wunsch bound provided that $\eps<\frac{1}{8}$. 
It is natural that the smaller $\eps$ is the 
better the bound is. Indeed, the condition on the distance between $E$ and $\Sigma$ 
is comparable to a non-resonance condition and should imply heuristically that $u$ must have some mass in the wing region. 
It is quite interesting 
to have a quantitative statement confirming this heuristics. 
We will actually give a more general statement concerning more general billiards 
(see Theorem \ref{thm:main}).
In particular we will consider billiards with smoother boundaries (see Def. \ref{sec:defbill}) 
disregarding the fact that these may not be ergodic. 
Here again we expect the bound to be better when the billiard becomes smoother and this statement 
is made quantitative in Theorem \ref{thm:main}.  

The method we propose relies on comparing the Dirichlet energy quadratic form with    
another quadratic form arising from the adiabatic ansatz presented in the numerical study of 
eigenfunctions by B\"acker-Schubert-Stifter \cite{BSS}. This adiabatic quadratic form 
has also appeared recently in the works of Hillairet-Judge \cite{HJ} in the study of 
the spectrum of the Laplacian on triangles. These two quadratic forms are 
close provided we do not enter too deeply into the wing region so that the non-concentration 
estimate really takes place in a neighbourhood of the rectangle that becomes smaller and smaller 
when the energy goes to infinity (see Sections \ref{sec:opt1} and \ref{sec:opt2}). 
Since the new quadratic form may be addressed using separation of variables, we will 
show precise one-dimensional control estimates and then use them to prove our results.  
We have separated these one dimensional estimates in an appendix since they may be 
of independent interest. Finally, we remark that the method can be applied to quasimodes with 
some caution (see Remark \ref{rk:quasi}) but there are no 
reasons to think that the bound we obtain is optimal.

{\sc Acknowledgments.}  
JLM was supported in part by a Hausdorff Center Postdoctoral Fellowship at the University of Bonn, 
in part by an NSF Postdoctoral Fellowship at Columbia University.  Also, JLM wishes to thank the 
MATPYL program which supported his coming to the University of Nantes, where this research began. 
LH was partly supported by the ANR programs {\em NONaa} and {\em Methchaos}.  
In addition, the authors wish to thank Andrew Hassell and Alex Barnett for very helpful conversations 
during the preparation of this draft.

\section{The setting}\label{sec:defbill}
Take $L$ a function defined on $[-B_0,B_1]$ with the following properties :
\begin{itemize}
\item[-] For non-positive $x,$ $L(x)\,=\,L_0 >0.$
\item[-] On $(0,B_1)$, $L$ is smooth, non-negative and non-increasing. 
\item[-] When $x$ goes to $B_1$, $L'$ has a negative limit (either finite or $-\infty$).
\item[-] For small positive $x$, we have the following asymptotic expansions:
\begin{eqnarray}
\label{eq:Lnear0}
\left\{ \begin{array}{c} 
L(x)  =  L_0 - c_L x^\gamma+o(x^\gamma) , \\
L' (x)  =  -c_L \gamma x^{\gamma-1} +o(x^{\gamma-1})  
\end{array} \right.
\end{eqnarray}
for some positive $c_L$ and $\gamma \geq 3/2$.  
\end{itemize} 

\begin{figure}[h]
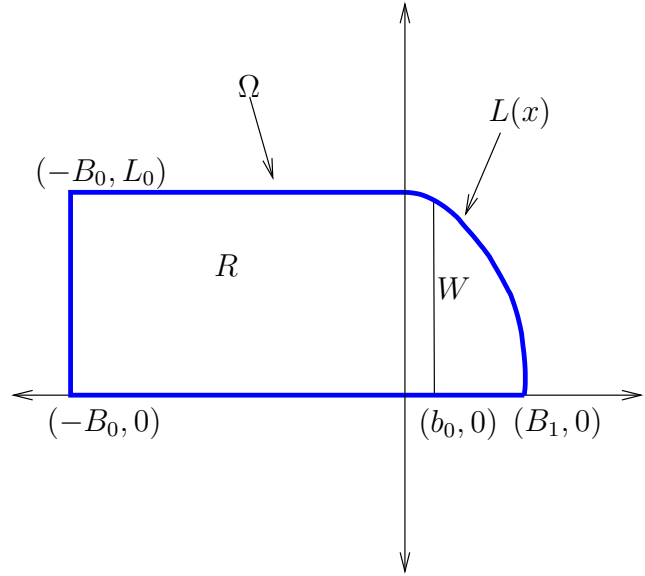

\include{billiard}
\caption{An example of the billiard $\Omega$.}
\label{fig:bil}
\end{figure}

The billiard $\Omega$ is then defined by 
\[
\Omega\,=\,\left \{ (x,y),~|~-B_0\leq x\leq B_1,~0\leq y\leq L(x)\,\right \}.
\]

See Fig. \ref{fig:bil} for an example of an applicable billiard.
For any $b < B_1$, we will denote by $\Omega_b:= \Omega \cap \{ x\leq b\}$ and by 
$W_b:=\Omega \cap \{ 0\leq x\leq b\}.$

We study eigenfunctions of the positive Dirichlet Laplacian, $\Delta$, on $\Omega$.  
Namely, we study solutions, $u_E$ such that
\begin{eqnarray*}
\Delta u_E \,=\,-\left( \partial^2_x\,+\,\partial_y^2\right)u_E\,=\, E u_E, \\
u_E |_{\p \Omega} = 0,
\end{eqnarray*}
where $E > 0$.  

We may formulate this equation using quadratic forms. We thus introduce $q$ defined on $H^1(\Omega)$ by 
\[
q(u)\,=\,\int_{\Omega} | \nabla u|^2 dxdy.
\]
The Euclidean Laplacian with Dirichlet boundary condition in $\Omega$ is the unique self-adjoint operator 
associated with $q$ defined on $H^1_0(\Omega).$   We denote by $q_b$ the restriction of $q$ to $H^1(\Omega_b)$ and by $\Delta_b$ the Dirichlet Laplace operator on 
$\Omega_b.$ We will also denote by $\D_b$ the set of smooth functions with compact support in $\Omega_b.$

\section{Adiabatic approximation}

Motivated by the well-known eigenvalue problem on a rectangular billiard and 
computational results in \cite{BSS}, we introduce a second family of quadratic forms 
$a_b$ and compare it to $q_b.$ 

For any $b < B_1$ and any $u\in \D_b,$ Fourier decomposition in $y$ implies that  
\begin{eqnarray}
\label{eqn:adiab}
u(x,y) = \sum_{k} u_k (x) \sin \left( \frac{\pi k}{L(x)} y \right).
\end{eqnarray}

Since 
\[
\int_0^{L(x)} |\sinL|^2dy = \frac{L(x)}{2}
\]
each Fourier coefficient $u_k$ is given by  
\[ 
u_k(x)\,=\, \left [ \frac{2}{L(x)}\right ] 
\int_0^{L(x)} u(x,y) \sin \left( \frac{\pi k}{L(x)} y \right) \,dy.
\]

For such $u$, we define 
\begin{gather*}
a_b(u) \,=\, \sum_{k\in \NN} \int_{-B_0}^b \left ( |{u}'_k(x)|^2 + \frac{k^2 \pi^2}{L^2(x)} |u_k(x)|^2\right )\,\frac{L(x)}{2}dx,\\
N_b(u) \,=\, \sum_{k\in \NN} \int_{-B_0}^b |u_k(x)|^2\,\frac{L(x)}{2} dx.
\end{gather*}

Observe that for each fixed $x$, Plancherel's formula reads 
\[ 
\sum_{k\in \NN}  |u_k(x)|^2\,\frac{L(x)}{2}  \,=\,\int_0^{L(x)} |u(x,y)|^2 \,dy, 
\]
so that we get $N_b(u) = \| u\|^2_{L^2(\Omega_b)}$ by integration with respect to $x$.

Fixing some $0< b_0< B_1$, and using that $L$ is uniformly bounded above and below on $[-B_0,b_0]$ we find a constant $C$ 
such that for any $b\leq b_0$ and $u\in L^2(\Omega_b) :$
\begin{equation}
 \label{eq:Plancherel}
C^{-1} \|u\|^2_{\Omega_b} \leq \sum_{k=1}^\infty \| u_k\|_{L^2(-B_0,b)}^2 \leq C \|u\|^2_{\Omega_b}.
\end{equation}

The quadratic form $a_b$ appears as the direct sum of the following quadratic forms 
$a_{b,k}$ (that can be defined on the whole function space  $H^1(-B_0,b)$) :
\begin{equation}\label{eq:defnk}
a_{b,k}(u) :=  \int_{-B_0}^b \left ( |{u}'|^2 + \frac{k^2 \pi^2}{L^2(x)} |u|^2\right )\,\frac{L(x)}{2}dx.
\end{equation}

Recall that, on an interval $I,$ the standard $H^1$ norm is defined by 
\begin{equation}
  \label{eq:defH1norm}
  \|u\|_{H^1}:=\left( \|u'\|_{L^2(I)}^2+\|u\|_{L^2(I)}^2 \right)^{\und},
\end{equation}
so that, for any $k$ and $b<B_1$ and any $u\in \CI_0(-B_0,b)$ we have 
\begin{equation}\label{eq:equivnk}
\min \left(L(b)\,,\,\frac{k^2\pi^2}{L_0} \right) \| u\|_{H^1}^2 \,\leq\, a_{b,k}(u)\,\leq\, 
\max \left( L_0\,,\,\frac{k^2\pi^2}{L(b)} \right) \| u\|_{H^1}^2 .
\end{equation}
The norm $a_{b,k}^{\und}$ thus defines on $H^1(-B_0,b)$ a norm that is equivalent to 
the standard $H^1$ norm. 

\subsection{Comparing $a_b$ and $q_b$}
To compare $a_b$ and $q_b$, we introduce the following operators $D$ and $R$ defined on $\D_b$ by 
\begin{gather*}
Ru\,=\,\frac{y{L}'(x)}{L(x)}\partial_y u,\\
Du\,=\,\partial_x u+Ru.
\end{gather*}
 
Using Plancherel formula for each fixed $x$ and then integrating, we obtain 
\[ 
a_b(u)\,=\, \int_{\Omega_b} |Du|^2+ |\partial_y u|^2 dxdy .
\]
from which the following holds for any $u,v \in \D_b.$

\begin{eqnarray}
a_b(u,v)-q_b(u,v)&=& \langle Du,Dv\rangle \,-\,\langle \partial_x u, \partial_x v\rangle  \notag \\
& = &  \langle \partial_x u, Rv\rangle + \langle Ru, Dv\rangle ,  \label{eqn:aqcomp}\\
& = & \langle \partial_x u, Rv\rangle \,+\, \langle Ru, \partial_x v\rangle \,+\,\langle Ru, Rv \rangle \label{eqn:compaq}.
\end{eqnarray}

We thus obtain the following lemma.  

\begin{lem}
\label{lem:normequiv}
Let $\delta$ be the function defined by 
\[
\delta(b)\,=\, \sup_{(0,b]} |L'(x)|\,+\,\sup_{(0,b]} |L'(x)|^2.
\]
Then for all $u,v\in \D_b$ 
\begin{gather*}
\left | a_b(u,v)-q_b(u,v) \right | \,\leq \,\delta(b) \cdot q_b^\und(u)\cdot q^\und_b(v) .
\end{gather*}
\end{lem}

\begin{rem}
The function $\delta$ is continuous on $(0,B_1)$ and $\delta (b) =O(b^{\gamma-1})$ when 
$b$ goes to $0.$ 
\end{rem}

\begin{proof}
In \refeq{eqn:compaq}, we use Cauchy-Schwarz inequality, $\max(\| Du\|, \| \partial_y u\|)\leq a_b^\und(u)$, and the fact that 
$\frac{y}{L(x)}$ is uniformly bounded by $1$ on $\Omega.$ 
\end{proof}

The following corollary is then straightforward.
\begin{cor}\label{cor:defLambda}
For any $0<b<B_1$ and any $u\in H^1(\Omega)$, the linear functional $\Lambda$ defined by 
$\Lambda(v):= a_b(u,v)-q_b(u,v)$ belongs to $H^{-1}(\Omega_b).$ 
Moreover 
\[
 \| \Lambda \|_{H^{-1}(\Omega_b)} \,\leq \delta(b) \|u\|_{H^1(\Omega_b)}.
\]
\end{cor}

\section{Non-Concentration}
\subsection{Preliminary reduction}
Let $u$ be an eigenfunction of $q$ with eigenvalue $E$. 
And define the associated linear functional $\Lambda$ using corollary \ref{cor:defLambda}.

Integration by parts shows that for any $v\in H^1_0(\Omega_b)$ we have 
\[
q_b(u,v)\,=\, E\cdot \langle u,v \rangle_{L^2(\Omega)},
\]
so that we have 

\begin{equation}\label{eq:quasi}
a_b(u,v)-E\cdot N_b(u,v) \,=\, \Lambda(v).
\end{equation}

We now deal with this equation using the adiabatic decomposition. 
We thus define $\Lambda_k$ as the distribution over $\Dcal_b$ such that, 
for any $v\in \Dcal_b$, 
\begin{equation}
 \label{eq:defLambdak} 
\Lambda_k(v) := \Lambda\left( v(x)\sinL \right).
\end{equation}

\begin{rem}
From now on, $u$ will always denote the eigenfunction that we are dealing with. 
We will denote by $u_k$ the functions entering in the adiabatic 
decomposition of $u,$ by $\Lambda$ the linear functional associated with $u$ and by 
$\Lambda_k$ the one-dimensional linear functionals that are associated with $\Lambda$.
\end{rem}

A straightforward computation yields, that for any $v\in \Dcal_b$ we have 
\[ 
a_{b,k}(u_k,v)-E\cdot \int_{-B_0}^b u_k(x) v(x) \frac{L(x)}{2} dx \,=\,\Lambda_k(v),
\]
where $a_{b,k}$ is the quadratic form defined in \refeq{eq:defnk}.

An integration by parts then shows that, in the distributional sense in $(-B_0,b)$, we have  
\begin{equation}\label{eq:ODEk}
-\frac{1}{L}\frac{d}{dx}\left( L u'_k \right )\,+\, \left( \frac{k^2\pi^2}{L^2} -E\right ) u_k \,=\, \tilde{\Lambda}_k,
\end{equation}
where the linear functional $\tilde{\Lambda}_k$ is defined by 
\begin{equation}\label{eq:deftildeLambda}
 \tilde{\Lambda}_k(v):= \Lambda_k \left( \frac{2}{L}\cdot v \right).
\end{equation}

\begin{rem}\label{rk:prefactor}
Since $L$ is not smooth, this definition of $\tilde{\Lambda}_k$ doesn't make sense as a distribution.
However, in the next section, we will prove that $\Lambda_k$ actually is in $H^{-1}$ and, since 
multiplication by $\frac{2}{L}$ is a bounded operator from $H^{-1}(-B_0,b)$ into itself, we thus get 
that $\tilde{\Lambda}_k$ is a perfectly legitimate element of $H^{-1}$. 
Moreover, for any $b_0$ there exists $C(b_0)$ such that for any $b\leq b_0,$ and $v\in \Dcal_b$, we have 
\[
 \| \frac{2}{L}v\|_{H^{-1}(-B_0,b)}\,\leq \,C(b_0)\| v\|_{H^{-1}(-B_0,b)}.
\]
\end{rem}

We denote by $P_k$ the operator that is defined by 
\[
P_k(u)\,=\, -\frac{1}{L}\frac{d}{dx}\left( L u' \right )\,+\, \left( \frac{k^2\pi^2}{L^2} -E\right ) u,
\]
and we try to analyze the way a solution to equation \refeq{eq:ODEk} on $(-B_0,b)$ 
may be controlled by its behaviour on $(0,b)$. 

The strategy will depend upon whether $k$ is large or not, but first we have to get a bound 
on $\Lambda_k$ in some reasonable functional space of distributions.

\subsection{Bounding $\Lambda_k$}
In this section, we prove that each $\Lambda_k$ is actually in $H^{-1}(-B_0,b)$ and provide a bound 
for its $H^{-1}$ norm. 

We first note that, using \refeq{eq:defH1norm}, for any 
$F\in H^{-1}(-B_0,b) ~:$ 
\begin{eqnarray}
\label{eqn:H1equiv}
\| F\|_{H^{-1}(-B_0,b)}\,:= \sup_{\phi \in \D_b} 
\frac{ \left| F(\phi) \right|}{\| \phi\|_{H^1}} \leq \sup_{\phi \in \D_b} 
\frac{ \left| F(\phi)  \right|}{\| \phi'\|_{L^2}}.
\end{eqnarray}

Using \refeq{eqn:aqcomp} in the definition of $\Lambda_k$ -see \refeq{eq:defLambdak}- we obtain 

\[
 \Lambda_k(v)\,=\,\langle \partial_x u, R \left( v(x)\sinL\right) \rangle \,+ \,\langle R u, D \left( v(x)\sinL\right) \rangle. 
\]

Denote by $A_k(v)$ the first term on the right and $B_k(v)$ the second term. By inspection, we have 

\begin{gather*}
A_k(v) := \frac{k\pi}{2}\int_{0}^b v(x) \frac{L'(x)}{L(x)} F_k(x) \,dx\\
B_k(v) := \frac{1}{2}\int_0^b v'(x) L'(x) G_k(x) \, dx, 
\end{gather*}
where we have set 
\begin{gather}
\label{eq:defFk} F_k(x) := \frac{2}{L(x)}\int_0^{L(x)} 1_W \cdot y\partial_x u(x,y) \cdot \cosL \, dy \\
\label{eq:defG_k} G_k(x) := \frac{2}{L(x)} \int_0^{L(x)} 1_W \cdot y\partial_y u(x,y) \cdot \sinL \, dy 
\end{gather}

Since $u\in H^1(\Omega),$ $F_k$ and $G_k$ are $L^2(0,b)$ and we can estimate the $H^{-1}$ norm of $\Lambda_k$ using 
them. 
 
\begin{lem}
\label{lem:Hminus1bd}
For any $b_0<B_1,$ and given $\Lambda_k$ and $F_k,\,G_k$ defined as above, there exists $C = C(\Omega_{b_0})$
\begin{equation}\label{eq:estLambdak}
\| \Lambda_k \|_{H^{-1}} \leq C ( k b^{\gamma} \| F_k \|_{L^2 (0,b)} + b^{\gamma-1} \| G_k \|_{L^2 (0,b)}).
\end{equation}
\end{lem}

\begin{proof}
We estimate $A_k(v),$ using first an integration by parts 
\[
 A_k(v) := -\frac{k\pi}{2}\int_{-B_0}^b v'(x) \left(\int_{0}^x \frac{L'(\xi)}{L(\xi)}F_k(\xi) \, d\xi\right)\,dx. 
\]
Using Cauchy-Schwarz inequality, and the fact that $L'(\xi)=O(\xi^{\gamma-1})$ we have 
\[
 \left | \int_{0}^x \frac{L'(\xi)}{L(\xi)}F_k(\xi) \, d\xi \right | \,\leq \, C x_+^{\gamma-\frac{1}{2}} \|F_k\|_{L^2(0,b)}.
\]
Inserting into $A_k(v)$ and using Cauchy-Schwarz inequality again we get 
\[
 \left| A_k (v)\right | \leq C \cdot (k b^{\gamma})\|F_k\|_{L^2(0,b)}\cdot \| v' \|_{L^2(-B_0,b)}, 
\]
which gives the claimed bound using \refeq{eqn:H1equiv}.  

The second term is estimated using directly Cauchy-Schwarz estimate and the fact that $\sup_{[0,b]} |L'(x)| \leq C b^{\gamma-1}.$
We get 
\[
 |B_k(v)| \, \leq \,C \cdot b^{\gamma-1} \|G_k\|_{L^2(0,b)}\cdot \| v'\|_{L^2(-B_0,b)}. 
\]
That gives the claimed bound using again \refeq{eqn:H1equiv}.
\end{proof}

Define $F:= 1_W \partial_x u$ and $G:= 1_W \partial_y u.$ 
By definition, $F_k(x)$ is the Fourier coefficient of the function $F(x, \cdot)$ with respect to 
the Fourier basis $\left( y\mapsto \cosL \right)_{k\in \NN\cup \{0\}}.$ 

Using Plancherel formula we get  
\[
 \sum_{k\geq 1} F_k(x)^2 \frac{L(x)}{2}  \leq \int_0^{L(x)} |F(x,y)|^2 dy. 
\]
For the same reason, but using this time the $\sin$ basis we have 

\[
 \sum_{k\geq 1} G_k(x)^2 \frac{L(x)}{2} \,=\, \int_0^{L(x)} |G(x,y)|^2 dy. 
\] 

Integrating with respect to $x$ and bounding $y$ from above and $L(x)$ from below uniformly we get the following lemma :

\begin{lem}\label{lem:boundFG}
For any $b_0$ there exists $C$ depending only on the billiard and $b_0$ such that, for any $b<b_0,$ 
\begin{gather}
 \label{eq:boundsumFk} \sum_{k\geq_1} \|F_k \|^2_{L^2(0,b)} \leq C \| \partial_x u \|^2_{L^2(W_b)}, \\
 \label{eq:boundsumGk} \sum_{k\geq_1} \|G_k \|^2_{L^2(0,b)} \leq C \| \partial_y u \|^2_{L^2(W_b)}.
\end{gather}
\end{lem}

We now switch to the control estimate. We begin by dealing with the modes for 
which $\frac{k^2\pi^2}{L_0^2}-E \geq E$.

\subsection{Large modes}
\subsubsection{A control estimate}  
Equation \refeq{eq:ODEk} may be rewritten 

\begin{equation}
  \label{eq:Schrolm}
-u_k'' \,+\,\left( \frac{k^2\pi^2}{L^2(x)}-E \right )u_k \,=\,h_k,
\end{equation}
where $h_k$ is the element of $H^{-1}$ defined by 
\begin{equation}
  \label{eq:defhklm}
h_k := \tilde{\Lambda}_k + \frac{L'}{L}u_k'  
\end{equation}

The $H^{-1}$ norm of $h_k$ is now estimated using the following

\begin{lem}\label{lem:normhklm}
There exists a constant $C:=C(b_0)$ such that for any $b\leq b_0$ and any $k$ such that 
$\frac{k^2\pi^2}{L_0^2}-E\geq E$ the following estimate holds:
\begin{equation}
  \label{eq:normhklm}
  \| h_k\|_{H^{-1}(-B_0,b)}\,\leq \, C(b_0)\left(   k b^{\gamma} \| F_k \|_{L^2 (0,b)} + b^{\gamma-1} \| G_k \|_{L^2 (0,b)}  \,+\,b^{\gamma-1}\|{u_k}\|_{L^2(0,b)} \right).
\end{equation}
\end{lem}
\begin{proof}
Using remark \ref{rk:prefactor}, the norm of $\tilde{\Lambda}_k$ is uniformly controlled by the norm of $\Lambda_k$ and 
the latter is estimated using Lemma \refeq{lem:Hminus1bd}. 
To estimate the $H^{-1}$ norm of $\frac{L'}{L}u'_k,$ 
we first set $v= \frac{L'}{L} u_k',$ and remark that 
\begin{eqnarray*}
v= \left( \frac{L'}{L} u_k \right)' -
\left(\frac{L''}{L}-\frac{(L')^2}{L^2}\right)u_k.
\end{eqnarray*}

We choose a test function $\phi$ and estimate 
\[
I_1\,=\,\int_{-B_0}^b \left( \frac{L'}{L}u_k \right)' \phi\,dx.
\]
We perform an integration by parts, use that 
$\frac{L'(x)}{L(x)}\leq C b^{\gamma-1}1_{x>0}$, 
then apply the Cauchy-Schwarz inequality to get 
\begin{eqnarray*}
\left| \int_{-B_0}^b \left( \frac{L'}{L}u_k \right)' \phi \,dx \right | & =& 
\left| \int_{-B_0}^b \frac{L'(x)}{L(x)}u_k (x) \phi'(x)\,dx\right |\\
&\leq & Cb^{\gamma-1}\|u_k \|_{L^2(0,b)}\| \phi'\|_{L^2(-B_0,b)}.
\end{eqnarray*}

We then estimate
\[
I_2\,=\,\left | \int_{-B_0}^b\left(\frac{L''(x)}{L(x)}-
\frac{(L'(x))^2}{L^2(x)}\right)u(x)\phi(x)\,dx\right|.
\]
We perform an integration by parts, use that 
\[
\left|\left(\frac{L''(x)}{L(x)}-\frac{(L'(x))^2}{L^2(x)}\right)\right|\,\leq\, C x_+^{\gamma-2}, 
\]  
then twice apply the Cauchy-Schwarz inequality to get 
\begin{eqnarray*}
I_2&\leq & C \int_{-B_0}^b \left(\int_0^x \xi_+^{\gamma-2}|u(\xi)|\,d\xi\right)|\phi'(x)|\, dx \\
&\leq & C b^{\gamma-1}\|u\|_{L^2(0,b)}\|\phi'\|_{L^2(0,b)}   .
\end{eqnarray*}
 The claim follows using \refeq{eqn:H1equiv}.
\end{proof}

The variational formulation of equation \refeq{eq:Schrolm} is given by 
\begin{equation}\label{eq:varlm}
\int_{-B_0}^b u_k'v' \,dx \,+\,\int_{-B_0}^b\left( \frac{k^2\pi^2}{L^2(x)}-E\right )u_kv \, dx \,=\,h_k(v).
\end{equation}

Since $\frac{k^2\pi^2}{L_0^2}-E \geq E$, the left-hand side of the preceding equation is a continuous 
quadratic form on $H^1_0(-B_0,b)$ so that, using Lax-Milgram theory, there exists a unique $v_k$ in 
$H^1_0(-B_0,b)$ that satisfies \refeq{eq:Schrolm} in the distributional sense.

The following lemma allows us to estimate the $L^2$ norm of this $v_k$.

\begin{lem}\label{lem:estvlm}
There exists a constant $C$ depending only on $b_0$ but 
neither on $b<b_0,$ $k$, nor $E$ such that, if $E\geq 1$ and $\frac{k^2\pi^2}{L_0^2}-E\geq E$ 
then the variational solution $v_k$ in $H^1_0(-B_0,b)$ to equation \refeq{eq:Schrolm} 
satisfies 
\begin{eqnarray}
  \label{eq:estvlm}
  \| v_k\|_{L^2(-B_0,b)}\,&\leq &\, C(b_0)\left( b^{\gamma} \| F_k \|_{L^2 (0,b)} +
 E^{-\und} b^{\gamma-1} \| G_k \|_{L^2 (0,b)}\,\right .\\
\nonumber & & \left. +\, E^{-\und}  b^{\gamma-1}\| u_k\|_{L^2(0,b)} \right).
\end{eqnarray}
\end{lem}

\begin{proof}
Since $v_k$ is a variational solution, putting $v=v_k$ in \refeq{eq:varlm} we get 
\begin{equation}\label{eq:varvk}
\int_{-B_0}^{b} |v_k'(x)|^2 dx \,+\,\int_{-B_0}^b \left(\frac{k^2\pi^2}{L^2(x)}-E\right)|v_k(x)|^2 dx \,=\,h_k(v_k).
\end{equation}
In the regime we are considering the second integral on the left is positive so that we obtain 
\[
\int_{-B_0}^{b} |v_k'(x)|^2 dx \,\leq\,|h_k(v_k)| \,\leq\,\|h_k\|_{H^{-1}}\|v_k\|_{H^1}.
\]
Since $v_k$ is in $H^1_0(-B_0,b),$ Poincar\'e inequality gives $c(b)$ a positive continuous 
function of $b$ such that 
\[
\int_{-B_0}^{b} |v_k'(x)|^2 dx \,\geq\, c(b) \|v_k\|^2_{H^1}.
\]
This gives a constant $C$ depending only on $b_0$ such that, for any $0<b<b_0,$ we have 
\[
\| v_k\|_{H^1}\leq C \| h_k\|_{H^{-1}}.
\]
We now use again equation \refeq{eq:varvk} to obtain 
\[
\left( \frac{k^2\pi^2}{L_0^2}-E\right) \int_{-B_0}^b |v_k(x)|^2 dx \,\leq\, \|h_k\|_{H^{-1}}\|v_k\|_{H^{1}} 
\,\leq\, C \|h_k\|_{H^{-1}}^2
\]
with the preceding bound.
Using estimate \refeq{eq:normhklm} we obtain 
\[
\left( \frac{k^2\pi^2}{L_0^2}-E\right)^\und \| v_k\|_{L^2(-B_0,b)}\,\leq\, C
\left(  k b^{\gamma} \| F_k \|_{L^2 (0,b)} + b^{\gamma-1} \| G_k \|_{L^2 (0,b)} +b^{\gamma-1}\| u_k\|_{L^2(0,b)}\right).
\]
We divide both sides by $\left( \frac{k^2\pi^2}{L_0^2}-E\right)^\und$. 
The coefficient in front of  $b^{\gamma} \| F_k \|_{L^2 (0,b)} $ is bounded by  using the fact that
\begin{eqnarray*}
\sup_{\frac{k^2\pi^2}{L_0^2} -E \geq E} \frac{k^2}{\frac{k^2\pi^2}{L_0^2}-E}&=&
\sup_{Z\geq E}\frac{L_0^2}{\pi^2}\left(1+\frac{E}{Z}\right)\\
&=& \frac{L_0^2}{\pi^2}\left(1+\frac{E}{E}\right).
\end{eqnarray*}   
For the two other terms, we use simply that $\frac{k^2\pi^2}{L_0^2}-E\geq E.$
This gives the lemma.
\end{proof}

We can now let $w_k= u_k-v_k.$ By construction, $w_k$ is a solution 
to the homogeneous equation 

\begin{equation}\label{eq:Homlm}
-w''+ \left(\frac{k^2\pi^2}{L^2(x)}-E\right)w = 0.
\end{equation}

Moreover, since both $u_k$ and $v_k$ satisfy Dirichlet boundary condition 
at $-B_0$ we have that $w_k(-B_0)=0.$ 

Since the 'potential' part in equation \refeq{eq:Homlm} is bounded 
below by $E$, concentration properties of solutions may be 
obtained using convexity estimates.   

\begin{lem}\label{lem:wlm}
For any $b\leq b_0$, any solution $w$ to \refeq{eq:Homlm} 
such that $w (-B_0)=0$  satisfies 
\[
b\cdot \int_{-B_0}^b |w|^2(x)\, dx \leq (B_0+b_0)\int_0^b |w|^2(x) \, dx .
\]
\end{lem}

\begin{proof}
Multiplying the equation by $w$ we find 
\[
- w''w + \left(\frac{k^2\pi^2}{L^2(x)}-E\right) w^2 =0.
\]
It follows that 
\[
(w^2)'' \geq \beta^2 w^2, 
\]
for some positive $\beta$ (here $\beta^2\,=\,2E$).

Since $w (-B_0)=0$, using the maximum principle on $[-B_0,\xi]$, we obtain 
for all $ -B_0\leq x \leq \xi \leq b_0$
\[
w^2(x) \leq w^2(\xi) \frac{\sinh(\beta(x+B_0))}{\sinh(\beta(\xi+B_0))}.
\] 
For any $t\in [0,1],$ define $x(t)\,=\,-B_0+ t(B_0+b)$ and $\xi(t)=t b.$ 
Since for any $t$ we have $-B_0\leq x(t) \leq \xi(t) \leq b_0$, we may integrate the 
preceding relation :
\[
\int _0^1 w^2(x(t)) \, dt \leq \int_0^1 w^2(\xi(t)) 
\frac{\sinh(\beta(x(t)+B_0))}{\sinh(\beta(\xi(t)+B_0))} \, dt.
\]
Since $\sinh$ is increasing the quotient of $\sinh$ is bounded above by $1$ and we obtain 
\[
b\cdot \int_{-B_0}^b w^2(x) \,dx \leq (B_0+b)\int_{0}^b w^2(x) \, dx.
\]
\end{proof}

Putting these two lemmas together we obtain the following

\begin{prop}\label{prop:controllm}
There exists a constant $C$ depending only on $b_0$ 
such that for any $b\leq b_0$, for any $k$ and $E$ 
such that $k^2\pi^2/L_0^2-E \geq E$ and $E\geq 1$
\begin{eqnarray}
\label{eq:controllm}
\| u_k\|_{L^2(-B_0,b)}  & \leq &  C \left( b^{\gamma-\und} \| F_k \|_{L^2 (0,b)} + E^{-\und} b^{\gamma-\frac{3}{2}} \| G_k \|_{L^2 (0,b)} \right. \\
&& \left. + b^{-\und}\| u_k\|_{L^2(0,b)}  \right) \notag
\end{eqnarray}
for $C = C(b_0)$.
\end{prop}

\begin{proof}
According to Lemma \ref{lem:wlm}  we have 
\[
\|w_k\|_{L^2(-B_0,b)} \leq C b^{-\und} \| w_k\|_{L^2(0,b)},
\]
where $w_k=u_k-v_k$ and $v_k$ is the variational solution constructed above.
Using the reverse triangle inequality, we obtain
\begin{eqnarray*}
\| u_k \|_{L^2(-B_0,b)} \leq C b^{-\und} \|u_k\|_{L^2(0,b)} + 
(C+b_0^\und) b^{-\und} \| v_k \|_{L^2 (-B_0,b)} .
\end{eqnarray*}
The claim will follow using estimate \refeq{eq:estvlm} of Lemma \ref{lem:estvlm}.
Observe that the prefactor of $\|u_k\|_{L^2(-B_0,b)}$ is at first (up to a constant prefactor) 
\[
 b^{-\und}\,+\,b^{-\und}E^{-\und}b^{\gamma-1}.
\]
Since $E^{-\und}b^{\gamma-1}$ is uniformly bounded we obtain the given estimate.
\end{proof}

\subsubsection{Summing over $k$}
We will now sum the preceding estimates over $k$. We thus introduce 
\[
u_+(x,y)\,=\,\sum_{\frac{k^2\pi^2}{L_0}-E\geq E} u_k (x) \sin \left( \frac{k \pi y}{L(x)}\right)
\]
and prove the following proposition.

\begin{prop}\label{prop:estuplus}
There exists $b_0$ and $E_0$ and a constant $C$ depending only on $E_0$ and $b_0$ 
such that if $u$ is an eigenfunction with energy $E\,>\,E_0$ and $b<b_0,$ 
then :
\[
\| u_+\|_{L^2(R)}^2 \,\leq C \left[(b^{2\gamma-1} \| \partial_x u\|^2_{W_b} \,
+\,E^{-1}b^{2\gamma-3}\| \partial_y u \|^2_{L^2(W_b)} \,+\, b^{-1}\|u\|^2_{L^2(W_b)}  \right] .
\]
\end{prop}

\begin{proof}
We square estimate \refeq{eq:controllm}, sum over $k,$ and use \refeq{eq:Plancherel} and Lemma \ref{lem:boundFG}. 
\end{proof}

Observe that the controlling term in the preceding estimate is supported in the wing region. However, 
compared to the usual bounds (such as in \cite{BHW}) there is a loss of derivatives since 
we need $\partial_x u$ and $\partial_y u$ in the wings. 

We also obtain the following corollary.

\begin{cor}\label{cor:estuplus}
Let $b_0$ and $E_0$ be fixed. there exists $C$ depending on the billiard $b_0$ and $E_0$ but not 
on the eigenfunction nor on $b<b_0$ such that 
\[
\| u_+\|_{L^2(R)}^2 \,\leq C \left[ (b^{2\gamma-1}E\,+\,b^{2\gamma-3}) \| u \|^2_{L^2(\Omega)} \,+\, b^{-1}\|u\|^2_{L^2(W)}  \right].
\]
\end{cor}

\begin{proof}
We bound $\| \partial_x u\|^2_{L^2(W_b)}$ and $\| \partial_y u \|^2_{L^2(W_b)}$ by $E\|u\|^2_{L^2(\Omega)}$ and use 
the fact that the norm over $W_b$ is less than the norm over $W.$  
\end{proof}

It remains to choose $b$ in a clever way to obtain the desired bound. 
\subsubsection{Optimizing $b$}\label{sec:opt1}
We will choose $b$ to be of the form $M^{-1}E^{-\alpha}$ for some constants $M$ and $\alpha$ to be 
choosen. As long as $\alpha$ is positive, there is some large $E_0$ such that for any 
$E\geq E_0$ then $b= ME^{-\alpha}<b_0$ so that we can use the preceding proposition. 

We obtain 
\begin{equation}\label{eq:optuplus}
\| u_+\|_{L^2(R)}^2 \,\leq C \left[ (M^{1-2\gamma}E^{1-\alpha(2\gamma-1)}\,+\,E^{-\alp (2\gamma-3)}) \| u\|^2_{L^2(\Omega)}\,
+\,ME^{\alpha} 
\|u\|_{L^2(W)}^2\right].
\end{equation}
 
It remains to make good choices to obtain the following proposition. 

\begin{prop}\label{prop:lmfinal}
There exists $E_0$ and $C$ depending only on the billiard such that for any $u$ 
eigenfunction with energy $E>E_0$ the following holds :
\begin{equation}\label{eq:lmfinal}
\|u_+\|^2_{L^2(R)}\,\leq\, \frac{1}{4} \|u\|^2_{L^2(\Omega)}\,+\,
C E^{\frac{1}{2\gamma-1}} \|u\|^2_{L^2(W)}         
\end{equation}
\end{prop}

\begin{proof}
We choose $\alpha:= \frac{1}{2\gamma-1}$ and $M$ such that $CM^{1-2\gamma}\,=\,\frac{1}{8}$.  For $E$ large enough, 
$E^{-\alp(2\gamma-3)}$ goes to zero. It is thus bounded by $\frac{1}{8C}$ for $E$ large enough. 
Replacing in \refeq{eq:optuplus} we get :
\[
\| u_+\|^2_{L^2(R)}\,\leq\, \frac{1}{4} \|u\|^2_{L^2(\Omega)}\,+\,
CE^{\frac{1}{2\gamma-1}}
\|u\|_{L^2(W)}^2.
\]
The claim follows.
\end{proof}

\subsection{Small modes} 

We now consider modes for which $\frac{k^2\pi^2}{L_0^2}-E \leq E,$ 
and this time we rewrite the equation 
$P_k(u_k)=\Lambda_k$ in the following form:
\begin{equation}\label{eq:Schrosm}
- u_k'' \,-\,z_k u_k \,=\,h_k,  
\end{equation}
in which we have set $z_k\,:= E-\frac{k^2\pi^2}{L_0^2}$ and 
\[
h_k :=\,\tilde{\Lambda}_k\,+\, \frac{L'}{L}{u}'_k + 
\frac{k^2}{\pi^2}\left (\frac{1}{L_0^2}-\frac{1}{L^2}\right ) u_k. 
\]

\subsubsection{The control estimate}
Since $z_k\geq -E$ we can use the results of the appendix 
to control $\|u_k\|_{L^2(-B_0,b)}.$

To do so, we need to estimate the norm of $h_k$ in $H^{-1}(-B_0,b)$.  

\begin{lem}\label{lem:estnormhsm}
There exists some constant $C$ depending only on $b_0$ such that, 
for any $b\leq b_0$ and  any $k$ such that $\frac{k^2\pi^2}{L_0^2}-E \leq E,$ 
the following holds :
\begin{eqnarray}
  \label{eq:estnormhsm}
  \| h_k\|_{H^{-1}(-B_0,b)}&\leq & C \left(  k b^{\gamma} \| F_k \|_{L^2 (0,b)} + b^{\gamma-1} \| G_k \|_{L^2 (0,b)} \right.\\ 
\nonumber & & +\left . \,(b^{\gamma-1}\,+\,k^2b^{\gamma+1})  \|u_k\|_{L^2(0,b)}\right).
\end{eqnarray}
\end{lem}

\begin{proof}
We use the definition of $h_k\,=\, \frac{2}{L(x)}\Lambda_k \,+\, \frac{L'}{L}{u}'_k + 
\frac{k^2}{\pi^2}\left (\frac{1}{L_0^2}-\frac{1}{L^2}\right ) u_k$ and estimate 
each term separately. 
The first term is estimated using \refeq{lem:Hminus1bd} and remark \ref{rk:prefactor}.
The second term is estimated as in the proof of lemma \ref{lem:normhklm}.

The same method applies to estimate the third term. 
We introduce   
\[
I_3\,=\,\left| \int_{-B_0}^b \left (\frac{1}{L_0^2}-\frac{1}{L^2(x)}\right ) u_k(x)\phi(x) \,dx\right|, 
\]
and observe that $(\frac{1}{L_0^2}-\frac{1}{L^2(x)})$ is $O(x_+^\gamma)$. 
Integrating by parts and using twice Cauchy-Schwarz inequality gives 
\[
I_3\,\leq \, C b^{\gamma+1} \| \phi' \|_{L^2 (0,b)} \|u_k\|_{L^2(0,b)} .
\]

Using the definition of the $H^{-1}$ norm (see \refeq{eqn:H1equiv}) and putting these estimates together 
yield the lemma.         
\end{proof}

For any $E\in \RR,$ denote by $\nu(E):=\min\left\{ \left| E - \frac{k^2\pi^2}{L_0^2}-\frac{l^2\pi^2}{B_0^2} \right|,~
(k,l)\in \NN \times \NN \right \}.$ 

\begin{rem}
Observe that $\nu(E) \leq \min\left\{ \left| E - \frac{k^2\pi^2}{L_0^2}-\frac{\pi^2}{B_0^2} \right|,~
k \in \NN \right \},$ so for $E$ large, we have
\begin{eqnarray}
\label{eqn:nubound}
 \nu(E) < c \sqrt{E}
 \end{eqnarray}
for some constant $c$.
\end{rem}

\begin{lem}\label{lem:estsin}
For any $\beta>0$, there exists some $c$ such that the following holds. 
For any $k$ such that $z_k\,=\,E-\frac{k^2\pi^2}{L_0^2}\,\geq \beta^2$
\[
|\sin(B_0\sqrt{z_k})|\geq c \cdot \frac{\nu(E)}{\sqrt{z_k}}.
\]
\end{lem}

\begin{proof}
First we use that there exists some $c$ such that 
\[
\forall x\in\RR,~~~~\left| \sin x\,\right|\geq c\cdot \mathrm{dist}(x,\pi\ZZ).
\]
We denote by $l_k$ the integer such that 
\[
\mathrm{dist}\left(\sqrt{z_k},\frac{\pi}{B_0}\ZZ\right )\,=\,\left|\sqrt{z_k}-\frac{l_k\pi}{B_0}\right |,
\]
so that we have 
\begin{eqnarray*}
\left|\sin(B_0\sqrt{z_k})\right|&\geq &c \cdot \left|\sqrt{z_k}-\frac{l_k\pi}{B_0}\right | \\
&\geq& c \cdot \frac{z_k-\frac{l_k^2\pi^2}{B_0^2}}{\sqrt{z_k}+\frac{l_k\pi}{B_0}}\\
&\geq & c \cdot \frac{E-\frac{k^2\pi^2}{L_0^2}-\frac{l_k^2\pi^2}{B_0^2}}{\sqrt{z_k}},
\end{eqnarray*} 
where, for the last bound, we have used the lemma \ref{lem:distla} below.

The claim follows by definition of $\nu(E).$
\end{proof}

\begin{lem}\label{lem:distla}
Fix $\alpha>0$ and denote by $l$ the (step-like) function on $[0,\infty)$ that is defined 
by 
\[
|\lambda-l(\lambda)\alpha|\,=\, \mathrm{dist}(\lambda,\alpha\ZZ)
\]  
Then there exists some $C$ such that 
\[
\forall \lambda\in [0,\infty),~~ \lambda+l(\lambda)\alpha \leq C\lambda.
\] 
\end{lem}
\begin{proof}
Define $f$ by $f(\lambda)\,=\,\frac{\lambda+l(\lambda)\alpha}{\lambda}.$ 
First, since $l$ vanishes on  $[0,\frac{\alpha}{2}],$ we have $f(\lambda)=1$ 
on this interval. Second, the function $f$ tends to the limit $2$ when 
$\lambda$ goes to infinity. Finally,  on $[\frac{\alpha}{2},M]$ we have 
$f(\lambda)\,=\,1+\frac{l(\lambda)}{\lambda} \alpha \leq 1+ \frac{2M+1}{\alpha}.$ 
\end{proof}

Putting these estimates together, we get the following
\begin{prop}
  \label{prop:estsm}
There exists $b_0$ and $E_0$ and a constant $C:=C(b_0,E_0)$ such that the following 
holds. For any $E>E_0,$ 
for any $k$ such that $\frac{k^2\pi^2}{L_0^2}-E\leq E$ and for any $b<b_0$, we have the following 
estimate
\begin{eqnarray}
  \label{eq:estsm}
  \| u_k\|_{L^2(-B_0,b)} & \leq & C \frac{E^\und}{\nu(E)}
\left[ E^\und b^{\gamma+\und}\| F_k \|_{L^2 (0,b)} + b^{\gamma-\und} \| G_k \|_{L^2 (0,b)} \right. \\
\nonumber && \left. +\left( 1\,+\,Eb^{\gamma+2}\right) b^{-\und}\| u_k \|_{L^2(0,b)}\right].
\end{eqnarray}
\end{prop}

\begin{proof}
For any $k$ we let $z_k= E-\frac{k^2\pi^2}{L_0^2}$ and use the estimates of 
the appendix combined with the bound on $h_k$ given by Lemma \ref{lem:estnormhsm}. 
For $k$ such that $z_k$ corresponds to estimates 
\refeq{eq:boundvzs1} and \refeq{eq:boundvzl1} of Theorem 
\ref{thm:controluosc} we obtain~:
\begin{eqnarray*}
\| u_k\|_{L^2(-B_0,b)}&\leq & C
\left[ b^{\und}\| h_k \|_{H^{-1}(-B_0,b)}\,+\,b^{-\und}\| u_k\|_{L^2(0,b)}\right ]\\
  & \leq &C
\left[ k b^{\gamma+\und} \| F_k \|_{L^2 (0,b)} + b^{\gamma-\und} \| G_k \|_{L^2 (0,b)} \right. \\
&& \left. +\left(b^{\gamma} +  k^2b^{\gamma+2} +1 \right) b^{-\und}\| u_k\|_{L^2(0,b)}\right] . \\
\end{eqnarray*}
We now use that $k = O(E^{\und})$ in the regime we are considering. 
We also remark that $\left( b^{\gamma} +  k^2b^{\gamma+2} +1 \right ) = O(1+Eb^{\gamma+2}).$

Otherwise (i.e. for $k$ such that $z_k$ corresponds to estimate \refeq{eq:boundvzi1}), 
we have to add a global 
$|\sin(B_0\sqrt{z_k})|^{-1}$ prefactor. 
Using Lemma \ref{lem:estsin}, we have 
\[
 |\sin(B_0\sqrt{z_k})|^{-1}\,\leq\, C \frac{\sqrt{z_k}}{\nu(E)}\,\leq \,C \frac{E^{\und}}{\nu(E)}. 
\] 

We thus obtain that for any $k$ the following holds :
\begin{eqnarray*}
\| u_k\|_{L^2(-B_0,b)} & \leq & C\cdot\max \left(1,\frac{E^\und}{\nu(E)} \right) \left[ E^{\und}b^{\gamma+\und} \| F_k \|_{L^2 (0,b)}\right. \\
&&  \left. + b^{\gamma-\und} \| G_k \|_{L^2 (0,b)}  +\left( 1+ Eb^{\gamma+2} \right)b^{-\und} \| u\|_{L^2(0,b)}\right].
\end{eqnarray*}

Using \eqref{eqn:nubound}, for large $E$ we have $\frac{E^{1/2}}{\nu(E)}$ is bounded from below, so that the claim follows. 
\end{proof}

\subsection{Summing over $k$}
We use the estimates of the preceding sections to obtain a control on $\|u_-\|_{L^2(R)}^2$ 
in which we have set 
\[
u_-(x,y)\,=\,\sum_{\frac{k^2\pi^2}{L_0}-E\leq E} u_k (x) \sin \left( \frac{k \pi y}{L(x)}\right).
\]

We prove the following proposition.

\begin{prop}\label{prop:estuminus}
There exists $b_0$ and $E_0$ and a constant $C$ depending only on $E_0$ and $b_0$ 
such that if $u$ is an eigenfunction with energy $E\,>\,E_0$ and $b<b_0,$ 
then
\begin{eqnarray*}
\| u_-\|_{L^2(R)}^2 \,&\leq& C \frac{E}{\nu(E)^2}\left [ Eb^{2\gamma+1} \| \partial_x u \|^2_{L^2(W)} + b^{2\gamma-1} \| \partial_y u \|^2_{L^2(W)}\right.\\
& & \left.+ \left( 1+ Eb^{\gamma+2} \right)^2 b^{-1}\|u\|^2_{L^2(W)} \right ].
\end{eqnarray*}
\end{prop}

\begin{proof}
We square \refeq{eq:estsm} and sum with respect to $k$. 
The Lemma \ref{lem:boundFG} controls $\sum \|F_k\|^2$ and $\sum \|G_k\|^2.$ Plancherel 
formula takes care of $\sum \|u_k\|^2$. We also use as before that the norm over $W_b$ 
is smaller than the norm over $W$. 
\end{proof}

As for the large mode case, we get a corollary using the fact that $\|\partial_x u\|^2$ and $\| \partial_y u\|^2$ 
are bounded above by $E\|u\|^2_{L^2(\Omega)}$.

\begin{cor}
 \label{cor:estuminus}
There exists $b_0$ and $E_0$ and a constant $C$ depending only on $E_0$ and $b_0$ 
such that if $u$ is an eigenfunction with energy $E\,>\,E_0$ and $b<b_0,$ 
then
\begin{eqnarray*}
\| u_-\|_{L^2(R)}^2 \,&\leq& C \left [ \left(\frac{E^3}{\nu(E)^2}b^{2\gamma+1} + \frac{E^2}{\nu(E)^2}b^{2\gamma-1}\right) \| u \|^2_{L^2(\Omega)}\right.\\
& & \left.+ \left( 1+ Eb^{\gamma+2} \right)^2 \frac{E b^{-1}}{\nu(E)^2}\|u\|^2_{L^2(W)} \right ].
\end{eqnarray*}
\end{cor}

\subsection{A non-resonance condition}
We now want to make the previous estimates explicit with respect to $E$ and $b$ so that 
we can use a similar optimization procedure as for the large modes case. We thus impose 
some condition on $\nu(E).$ Namely, for any $\eps \geq 0,$ we introduce the following set 

\begin{eqnarray*}
\Zcal_\eps&:=& \left\{ E\in \RR\,,|\, \nu(E) \geq c_0 E^{-\eps}\} \right .\\
&=& \left \{E\in \RR~|~\forall k,l\in \NN,~
|E-\frac{k^2\pi^2}{L_0^2}-\frac{l^2\pi^2}{B_0^2}|\geq c_0 E^{-\eps}\,\right\} .
\end{eqnarray*}

In other words, the set $\Zcal_{\eps}$ consists in energies that 
are far from the Dirichlet spectrum of the rectangle $[-B_0,0]\times[0,L_0].$ 
It is natural to say that such energies are {\em not resonating} with the 
rectangle. The coefficient $c_0$ which is irrelevent when $\eps>0$ has been chosen in 
such a way that Weyl's law for the rectangle implies that $\Zcal_0$ is not empty. 
Note however that, although expected, it is not clear that there actually are 
eigenvalues in $\Zcal_0,$ nor for that matter in $\Zcal_{\eps}$.
 
Once $\eps$ is fixed, the estimate of the corollary \ref{cor:estuminus} becomes :
\begin{equation}
\label{eq:estuminus}
\begin{split}
\| u_-\|_{L^2(R)}^2  \leq{} C\cdot \left[\right .&\left. \left( b^{2\gamma+1}E^{3+2\eps}\,+\,b^{2\gamma-1}E^{2+2\eps}\right)\| u\|^2_{L^2(\Omega)} \right. \\
& \left. + \left( 1+ Eb^{\gamma+2} \right)^2b^{-1}E^{1+2\eps} \|u\|_{L^2(W)}^2 \right] .
\end{split}
\end{equation}

\subsubsection{Optimizing $b$}\label{sec:opt2}

As before we let $b\,=\,ME^{-\alpha}$ for some positive $\alpha$ and try to optimize the 
bound.

\begin{prop}\label{prop:smfinal}
Define $\alp$ by 
\[
\alp= \max \left(\frac{3+2\eps}{2\gamma+1}, \frac{2+2\eps}{2\gamma-1}\right).
\]
There exists $E_0$ and $C$ such that for any $u$ 
eigenfunction with energy $E$ in $\Zcal_\eps$ such that $E>E_0,$ the following holds :
\begin{equation}\label{eq:smfinal}
\|u_-\|^2_{L^2(R)}\,\leq\, \frac{1}{4} \|u\|^2_{L^2(\Omega)}\,+\,
C\cdot E^{1+2\eps+\alp}\cdot \|u\|_{L^2(W)}^2    .     
\end{equation}
\end{prop}

\begin{proof}
With the given choice of $\alp$ it is possible to choose $M$ so that 
the prefactor of $\|u\|_{L^2(\Omega)}^2$ is $\frac{1}{4}$ for $E$ large enough. The claim follows remarking 
that the definition of $\alp$ implies $\alp \geq \frac{3}{2\gamma+1} \,>\, \frac{1}{\gamma+2}$ so that 
the prefactor $\left( 1+ Eb^{\gamma+2} \right)^2$ is uniformly bounded above.
\end{proof}

\section{Non-concentration Estimate}
We now put all the estimates together to obtain the following theorem.

\begin{thm}\label{thm:main}
Fix $\eps$, and define $\rho$ by 
\[
\rho:= \max \left (\frac{2+\gamma+2(\gamma+1)\eps}{2\gamma+1},\frac{1 + 2\gamma + 4\gamma\eps}{4\gamma-2}\right).
\]
There exists $E_0$ and $C$ such that any eigenfunction $u$ 
of $\Omega$ with energy $E$ in $\Zcal_{\eps}$ such that $E>E_0$ satisfies :
\[
\| u \|_{L^2(\Omega)}\leq C\cdot E^{\rho} \|u\|_{L^2(W)}.
\]
\end{thm}

\begin{proof}
We first remark that whatever the exponent $\alp$ is we always have 
$1\,+\,2\eps\,+\,\alp\geq 1 > \frac{1}{2\gamma-1}$ so that the exponent for the 
small modes is always larger than the exponent for the large modes. 
Thus, adding the estimates from propositions \ref{prop:lmfinal} and \ref{prop:smfinal}, 
we obtain 
\[
\| u\|_{L^2(R)}^2 \leq \frac{1}{2}\|u\|_{L^2(\Omega)}^2 \,+
\,CE^{1+2\eps+\alp}\|u\|_{L^2(W)}^2.
\]
Since $\|u\|^2_{L^2(R)}=\|u\|^2_{L^2(\Omega)}-\|u\|^2_{L^2(W)}$ we get 
\[
\frac{1}{2}\|u\|^2_{L^2(\Omega)} \leq 
(1+CE^{1+2\eps+\alp})\|u\|_{L^2(W)}^2 
\]
When $E$ is large the constant $1$ can be absorbed in the term with a power of $E$. 
The claim follows by computing $1+2\eps+\alp$ for both possible choices of $\alp$ 
and taking square roots.
\end{proof}

We state as a corollary the corresponding statement for the Bunimovich billiard 
(see theorem \ref{thm:intro}).
\begin{cor}
In the Bunimovich stadium, for any $\eps\geq 0$ there exists $E_0$ and $C$ such that 
if $u$ is an eigenfunction of energy $E$ in $\Zcal_{\eps}$ such that $E>E_0$ then the following estimate holds:
\[
\|u\|_{L^2(\Omega)}\,\leq C E^{\frac{5 + 8 \eps}{6}} \| u\|_{L^2(W)}.
\] 
\end{cor}

\begin{proof}
We let $\gamma=2$ so that $\alp=\max( \frac{4+6\eps}{5}, \frac{5+8\eps}{6}).$ For any non-negative $\eps$, 
$\frac{4+6\eps}{5} \leq \frac{5+8\eps}{6},$ this makes the proof complete.
\end{proof}

\begin{rem}\label{rk:improvBHW}
The bounds in \cite{BHW} gives a similar control with $1$ as the exponent of $E.$ 
Our bound thus gives a better estimate as long as $ \eps < \frac{1}{8}$.
As it has been recalled in the introduction, it is quite natural that the non-resonance condition 
allows to get better bounds.
\end{rem}

\begin{rem}\label{rk:quasi}
We could deal with quasimodes by adding an error term to $\Lambda$ that 
is controlled by some negative power of $E.$ There will be mainly two differences in the analysis. 
First the second term $\Lambda$ will not have support away from the rectangle anymore and 
second, in the optimization process, we will have to take care of the new error term 
(which will possibly change the range of applicable exponents).
\end{rem}
 
\begin{rem}
By adding the estimates in propositions \ref{prop:estuplus} and \ref{prop:estuminus}, we get a different control estimate, 
where the control still is in the wings but now with a loss in derivatives. We haven't tried to optimize 
this bound.   
\end{rem}

\begin{appendix}
  
\section{One dimensional Control Estimates}
The aim of this appendix is to provide a control estimate for the equation 
\[
-u''-z\cdot u \,=\, h
\]
on $[-B_0,b]$ of the following type 
\[
\|u\|_{L^2(-B_0,0)}\,\leq\, C_1\|h\|_{H^{-1}(-B_0,b)}\,+\,C_2\|u\|_{L^2(0,b)},
\]
in which we want an explicit dependence of the constants $C_1$ and $C_2$ 
on $z$ and $b$.  It is now standard (see \cite{BZ3}) that if $b$ is fixed then we can choose 
$C_1$ and $C_2$ to be independent of $z$ but what we need is an estimate 
when $b$ goes to $0.$ 

We first need a few preparatory lemmas. 

\begin{lem}\label{lem:vpsmallz}
For any $\eps>0,$ there exists a constant $C:=C(\eps)$ such that 
for any $b,$ for any $h\in H^{-1}(-B_0,b)$ and 
any $z$ such that $z\leq \frac{(1-\eps)\pi^2}{b^2}$, there exists 
a solution $v_p\in H^1_0(0,b)$ to
\[
-v_p'' -z v_p \,=\,h,
\]
in $\Dcal'(0,b)$ 
and 
\begin{equation}
\label{eq:vpsmallz}
\| v_p \|_{L^2(0,b)}\,\leq\, C b \|h\|_{H^{-1}(-B_0,b)}  .
\end{equation}   
\end{lem}

\begin{proof}
First we note that $h,$ when restricted to $(0,b)$ also belongs 
to $H^{-1}(0,b)$ and that $\|h\|_{H^{-1}(0,b)}\leq \|h\|_{H^{-1}(-B_0,b)}$.
The proof follows from a standard resolvent estimate since, on $(0,b),$  the 
bottom of the spectrum of the self-adjoint operator $v \mapsto -v''$ with 
Dirichlet boundary condition is $\frac{\pi^2}{b^2}.$ 
We include it for the convenience of the reader. 
We decompose $v_p$ in Fourier series :
\[
v_p(x)\,=\,\sum_{k\geq 1} a_k\sin(\frac{k\pi}{b}  x).
\]
We have
\begin{eqnarray*}
h(x)\,=\,\sum_{k\geq 1} \left( \frac{k^2\pi^2}{b^2}-z \right) a_k\sin(\frac{k\pi}{b}  x),
\end{eqnarray*}
hence
\begin{eqnarray*}
\|h\|_{H^{-1}(0,b)}^2\,=\,\sum_{k\geq 1} 
\frac{\left(\frac{k^2\pi^2}{b^2}-z\right)^2}{\frac{k^2\pi^2}{b^2}}|a_k^2|
\end{eqnarray*}
or
\begin{eqnarray*}
\|h\|_{H^{-1}(0,b)}^2 & \geq & \pi^2 b^{-2}
 \left[\inf_{k\geq 1} \frac{(1-\frac{zb^2}{k^2 \pi^2})^2}{\pi^2 k^2} \right]  \|v_p\|_{H^1(0,b)}^2 \\
 & \geq & c\pi^2 b^{-2} \|v_p\|_{L^2 (0,b)}^2.
\end{eqnarray*}
The claim follows since the inf is bounded away from zero in the regime we are 
considering.
\end{proof}

\begin{lem}\label{lem:homsol}
For any $z\leq \frac{(1-\eps)\pi^2}{b^2}$ 
Let $w\in H^1_0(-B_0,b)$ be a solution to 
\[
-w''-z w =0
\]
in $\Dcal '\left( (-B_0,b)\setminus \{0\}\right).$
Then there exists a constant $A$ such that $w\,=A  G,$ 
in which the function $G$ is defined by:
\begin{equation}\label{eq:defG}\displaystyle
G(x)\,=\,\left \{
\begin{array}{cr}
\frac{\sin(\sqrt{z}(x+B_0))}{\sqrt{z}} \frac{\sin(\sqrt{z}b)}{\sqrt{z}},& x<0,\\
\\
\frac{\sin(\sqrt{z}(b-x))}{\sqrt{z}}  \frac{\sin(\sqrt{z}B_0)}{\sqrt{z}},& x>0.\\
\end{array}
\right.
\end{equation}
\end{lem}
\begin{proof}
Let $w$ be such a function then necessarily there exist two constants $A_{\pm}$ such that 
\[
w(x)\,=\,\left \{
\begin{array}{cr}
A_- \frac{\sin(\sqrt{z}(x+B_0))}{\sqrt{z}}& x<0,\\
A_+ \frac{\sin(\sqrt{z}(b-x))}{\sqrt{z}}& x>0.\\
\end{array}
\right .
\]
By assumption $w \in H^1$ and hence is continuous at $0$, so
\[
A_- \frac{\sin(\sqrt{z}B_0)}{\sqrt{z}}\,=\,A_+ \frac{\sin(\sqrt{z}b)}{\sqrt{z}}.
\]
In the regime we are considering $\frac{\sin(\sqrt{z}b)}{\sqrt{z}}\neq 0$, hence
we can divide by this expression and express $A_-$ in terms of $A_+.$ 
The claim follows. 
\end{proof}

We finish these preparatory lemmas by establishing the control estimate 
for multiples of $G.$

\begin{lem}\hfill \\
  \label{lem:controlG}
\begin{enumerate}
\item 
For $\beta$ such that $0<\beta\leq \frac{\pi}{B_0},$ 
there exists $B_1\,=\,B_1(\beta)$ and $C:=C(\beta)$ such that 
for any $z\leq \beta^2$ and any $b<B_1$ the following estimate holds 
\begin{equation}
  \label{eq:controlGsm}
  \| G\|_{L^2(-B_0,0)}\,\leq\,C b^{-\und}  \|G\|_{L^2(0,b)}. 
\end{equation}
\item For any $\beta,\eps>0$ there exists $B_1:=B_1(\beta)$ and $C:=C(\beta,\eps)$ such that, for 
any $b\leq B_1$ and $\beta^2\leq z \leq \frac{(1-\eps)\pi^2}{b^2}$ the following estimate holds :
\begin{equation}
  \label{eq:controlGlm}
 \| G\|_{L^2(-B_0,0)}\,\leq\,C \frac{b^{-\und}}{\sin(\sqrt{z}B_0)} \|G\|_{L^2(0,b)}.
\end{equation}
\end{enumerate}
\end{lem}
\begin{proof}
In case $(1)$, we first assume that $z<-Z_0^2$ for some positive $Z_0.$ 
we set $z=-\omega^2$ and compute  
\begin{eqnarray*}
\int_{-B_0}^0 |G(x)|^2\,dx &=& \frac{\sinh^2(\omega b)}{\omega^2}\int_{-B_0}^0 \sinh^2(\omega(B_0+x))\,dx,\\
\int_{0}^b |G(x)|^2\,dx &=& \frac{\sinh^2(\omega B_0)}{\omega^2}\int_{0}^b \sinh^2(\omega(b-x))\,dx,\\
\end{eqnarray*}
By a straightforward change of variables we get 
\begin{eqnarray*}
\int_{-B_0}^0 |G(x)|^2\,dx &=& \frac{\sinh^2(\omega b)}{\omega^3}\int_{0}^{\omega B_0} \sinh^2(\xi)\,d\xi,\\
\int_{0}^b |G(x)|^2\,dx &=& \frac{\sinh^2(\omega B_0)}{\omega^3}\int_{0}^b \sinh^2(\xi)\,d\xi.\\
\end{eqnarray*}
We set $F(X):= \frac{\int_0^X \sinh^2(\xi)\,d\xi}{\sinh^2(X)}$ so that finally, we obtain :

\[
 \int_{-B_0}^0 |G(x)|^2\,dx\,=\, \frac{F(\omega B_0)}{F(\omega b)}\cdot \int_{0}^b |G(x)|^2\,dx. 
\]
It is straightforward that $F(X)$ is positive, tends to $1$ at infinity and that $F(X)/X$ tends to $\frac{1}{3}$ 
at $0.$ As a consequence, there exists some $C(Z_0)$ such that, for any $z<-Z_0^2$,
\[
 \int_{-B_0}^0 |G(x)|^2\,dx\leq C \max(1, (\omega b)^{-1}) \int_{0}^b |G(x)|^2\,dx, 
\]
For $b<B_1$ and $\omega >Z_0$, we have $\max(1, (\omega b)^{-1})\leq \max(1, \omega^{-1})b^{-1}\leq C b^{-1}$ 
which gives the claim for this range of parameters.
    
We now assume that we have $-Z_0^2 < z < \beta^2$. We  have 
\begin{eqnarray*}
\int_{-B_0}^0 |G(x)|^2\,dx&=& \left|\frac{\sin(\sqrt{z}b)}{\sqrt{z}}\right|^2 
\int_{-B_0}^0 \left|\frac{\sin(\sqrt{z}(x+B_0))}{\sqrt{z}}\right|^2\,dx\\ 
&=& b^2\left|\frac{\sin(\sqrt{z}b)}{\sqrt{z}b}\right|^2 
\int_{-B_0}^0 \left|\frac{\sin(\sqrt{z}(x+B_0))}{\sqrt{z}(x+B_0)}\right|^2(x+B_0)^2\,dx\\
&\leq & C b^2,
\end{eqnarray*}
where the constant $C$ comes from the fact that the function $\frac{\sin(w)}{w}$ 
is continuous and its argument belongs to a fixed compact set.
On the other hand, by a simple change or variables we have 
\begin{eqnarray*}
\int_{0}^b |G(x)|^2\,dx&=& \left|\frac{\sin(\sqrt{z}B_0)}{\sqrt{z}}\right|^2 
\int_{0}^b \left|\frac{\sin(\sqrt{z}x)}{\sqrt{z}}\right|^2\,dx \\
&\geq & c B_0^2 \frac{b^3}{3},
\end{eqnarray*}
in which $c$ is given by
\[
c\,=\,\left|\frac{\sin(\sqrt{z}B_0)}{\sqrt{z}B_0}\right|^2
\inf_{0\leq x\leq B_1}  \left|\frac{\sin(\sqrt{z}x)}{\sqrt{z}x}\right|^2.
\]
Using that $\frac{\sin(w)}{w}$ is continuous and does not vanish 
on $(-\infty, \pi)$ and choosing $B_1$ accordingly we obtain 
the first bound.\\
\hfill \\
For case $(2)$, we first use homogeneity and prove the bound for 
$ \tilde{G}\,:= z G.$  
We have
\begin{eqnarray*}
\int_{-B_0}^0 \left| \tilde{G}(x)\right|^2\,dx &=& \left|\sin(\sqrt{z}b)\right|^2 
\int_{-B_0}^0 \left|\sin(\sqrt{z}(x+B_0))\right|^2\,dx\\  
&\leq& B_0 \left|\sin(X)\right|^2,  
\end{eqnarray*}
in which we have set $X\,:=\sqrt{z}b.$
On the other hand we have 
\begin{eqnarray*}
\int_{0}^b \left| \tilde{G}(x)\right|^2\,dx &=& \left|\sin(\sqrt{z}B_0)\right|^2  
\int_{0}^b \left|\sin(\sqrt{z}x)\right|^2\,dx\\  
&=& b\cdot \left|\sin(\sqrt{z}B_0)\right|^2 \cdot \frac{1}{X}\int_{0}^X |\sin(\xi)|^2d\xi
\end{eqnarray*}
with the same $X.$
Under the assumptions, $X$ belongs to a compact subinterval of  $[0,\pi).$ 
Since on this interval the function 
$X \mapsto \frac{1}{X|\sin(X)|^2}\int_{0}^X |\sin(\xi)|^2d\xi$ is continuous, the claim 
follows. 
\end{proof}

We can now prove the following

\begin{prop}\label{prop:controlvosc}
There exist $\beta$ and $B_1:=B_1(\beta)$, such that 
if $b\leq B_1$ and $v\in H^1_0(-B_0,b)$ satisfies 
\[
-v''-z v\,=\,h,
\]
with $h$ that vanishes on $(-B_0,0),$   
then the following estimates hold:
\begin{enumerate}
\item If $z \leq \beta^2,$ then
  \begin{equation}
    \label{eq:boundvzs}
    \| v\|_{L^2(-B_0,0)}\,\leq\, C_1\left[ b^{\und} \|h\|_{H^{-1}(-B_0,b)}\,+\,b^{-\und}  \|v\|_{L^2(0,b)}\right],
  \end{equation}
\item If $\beta^2 \leq z \leq \frac{1}{b^2},$ 
 \begin{equation}
    \label{eq:boundvzi}
    \| v\|_{L^2(-B_0,0)}\,\leq\, C_1\left[ \frac{b^{\und}}{|\sin(B_0\sqrt{z})|} \|h\|_{H^{-1}(-B_0,b)}\,
+\,\frac{b^{-\und}}{|\sin(B_0\sqrt{z})|} \|v\|_{L^2(0,b)}\right],
  \end{equation}
\item If $\frac{1}{b^2}\leq z$ then 
\begin{equation}
    \label{eq:boundvzl}
    \| v\|_{L^2(-B_0,0)}\,\leq\, C_3\left[ b^{\und} \|h\|_{H^{-1}(-B_0,b)}\,+\,b^{-\und}  \|v\|_{L^2(0,b)}\right].
  \end{equation}
\end{enumerate}
\end{prop}
\begin{proof}
  In the first two cases, we have $z\leq \frac{1}{b^2}< \frac{\pi^2}{b^2}.$ 
We may thus consider $v_p$ as given by lemma \ref{lem:vpsmallz} and define $\tilde{v}_p$ by 
extending $v_p$ by $0$ for negative $x.$ Observe that 
$w:=v-\tilde{v}_p$ is in $H^1_0(-B_0,b)$ and satisfies 
\begin{eqnarray*}
-w''-z w=0
\end{eqnarray*}
in  $\Dcal((-B_0,b)\setminus \{ 0\})$ so that $v-\tilde{v}_p=A G$ 
for some $A$ according to Lemma \ref{lem:homsol}. 
Using Lemma \ref{lem:controlG} we obtain in the first case 
\[
\| v-\tilde{v}_p  \|_{L^2(-B_0,0)}\leq C b^{-\und}\|v-\tilde{v}_p\|_{L^2(0,b)}.
\]
We use the triangle inequality on the right-hand side and the fact that 
$\tilde{v}_p$ is $0$ for negative $x$ and coincide with $v_p$ for positive $v.$
We obtain 
\[
\|v\|_{L^2(-b_0,0)}\leq C b^{-\und}\left( \|v\|_{L^2(0,b)}\,+\,\|v_p\|_{L^2(0,b)}\right ).
\] 
The claim then follows from the estimate on $\|v_p\|_{L^2(0,b)}$ in lemma \ref{lem:vpsmallz}.
We prove the second case by following the same argument, inserting the corresponding 
bound for $G.$

The third case will follow the same lines but we will introduce a different 
particular solution $v_p,$ following then even more closely the proof of 
\cite{BZ3}. We set $\lambda=\sqrt{z}.$
   
Denote by $H$ the unique $L^2$ function on $(-B_0,b)$ that vanishes on $(-B_0,0)$ and such that 
$H'=h$ in the distributional sense. 
The $L^2$ norm of $H$ is related to the $H^{-1}$ norm of $h$ by the relation 
\[
\| H-\left(\int_0^b H(y)dy\right) \|_{L^2(-B_0,b)}\,=\, \|h\|_{H^{-1}(-B_0,b)}.
\]
The Cauchy-Schwarz inequality then implies that 
\begin{equation}\label{eq:compHh}
\| H\|_{L^2(-B_0,b)}\geq (1+b^{\und})^{-1} \|h\|_{H^{-1}(-B_0,b)}.
\end{equation}

Set  
\[
v_p(x)\,=\, \int_{-B_0}^x \frac{\sin(\lambda(x-y))}{\lambda} H'(y)\,dy,
\]
then $v_p$ satisfies 
\begin{eqnarray*}
-v_p''-\lambda^2 v_p\, =\, H'
\end{eqnarray*} 
in $\Dcal'(-B_0,b)$ and $v_p(-B_0)=0$ but the boundary condition need not be satisfied at $b.$ We thus have 
\[
v(x)\,=\, v_p(x)-v_p(b)\frac{\sin(\lambda(x+B_0))}{\sin(\lambda(B_0+b))}.
\]
The function $v-v_p$ is thus a multiple of $ \sin(\lambda(x+B_0)).$

We have 
\begin{eqnarray*}
\int_{-B_0}^0 \left|\sin (\lambda(x+B_0))\right|^2 \,dx \leq  B_0 
\end{eqnarray*}
and
\begin{eqnarray*}
\int_{0}^b \left|\sin (\lambda(x+B_0))\right|^2 \, dx \geq  \und(b -\frac{1}{2\lambda}).
\end{eqnarray*}
Hence, in the regime under consideration we have 
\begin{equation}\label{eq:vvplm}
\|v-v_p\|_{L^2(-B_0,b)}\leq C b^{-\und} \| v-v_p\|_{L^2(0,b)}.
\end{equation} 

We perform an integration by parts in $v_p$ and observe that the 
boundary contributions vanish because $H$ vanishes near $-B_0$ and 
$\sin(\lambda(y-x))$ vanishes at $y=x$.

Finally, we obtain
\[
v_p(x)\,=\, \int_{-B_0}^x \cos(\lambda(x-y))H(y) \,dy .
\] 
It follows that $v_p$ is identically $0$ on $(-B_0,0)$ and that, 
on $(0,b),$ it satisfies 
\begin{eqnarray*}
|v_p(x)|\leq \|H\|_{L^2(-B_0,b)} \sqrt{x}.
\end{eqnarray*}
Squaring and integrating, we get
\[
\|v_p\|_{L^2(0,b)}\leq b \|H\|_{L^2(-B_0,b)}.
\]
Using the triangle inequality in \refeq{eq:vvplm} and inserting this bound, 
the result follows for $b\leq \und$ using \refeq{eq:compHh}. 
\end{proof}

In the paper, we will need to relax the condition that $v(b)=0.$ 
This can be done using a standard construction related to a commutator 
method. We will get the following

\begin{thm}
  \label{thm:controluosc}
There exist $\beta$ and four  constants $B_1, C_1,C_2,C_3$ depending only on $\beta$ such that 
the following holds. For any $b\leq B_1$, 
for any function $u$ in  $H^1(-B_0,b)$ that satisfies 
\[
-u''-z u\,=\,h,
\]
with $h\in H^{-1}(-B_0,b)$ and such that $u(-B_0)=0$ and $h$ vanishes on $(-B_0,0)$.  Then, 
the following estimates hold:
\begin{enumerate}
\item If $ z \leq \beta^2,$ then
  \begin{equation}
    \label{eq:boundvzs1}
    \| u\|_{L^2(-B_0,0)}\,\leq\, C_1\left[ b^{\und} \|h\|_{H^{-1}(-B_0,b)}\,+\,b^{-\und} \|u\|_{L^2(0,b)}\right].
  \end{equation}
\item If $\beta^2 \leq z \leq \frac{1}{b^2}$, then
 \begin{equation}
    \label{eq:boundvzi1}
    \| u\|_{L^2(-B_0,0)}\,\leq\, C_1\left[ \frac{b^{\und}}{|\sin(B_0\sqrt{z})|} \|h\|_{H^{-1}(-B_0,b)}\,
+\,\frac{b^{-\und}}{|\sin(B_0\sqrt{z})|} \|u\|_{L^2(0,b)}\right].
  \end{equation}
\item If $\frac{1}{b^2}\leq z$, then 
\begin{equation}
    \label{eq:boundvzl1}
    \| u\|_{L^2(-B_0,0)}\,\leq\, C_3\left[ b^{\und} \|h\|_{H^{-1}(-B_0,b)}\,+\,b^{-\und} \|u\|_{L^2(0,b)}\right].
  \end{equation}
\end{enumerate}
\end{thm}

\begin{proof}
Define a smooth cutoff function $\rho_1$ such that $\rho_1(x)$ is identically 
$1$ if $x\leq \und$ and identically $0$ if $x\geq 1$ and let $\rho_b$ be the 
function $x\mapsto \rho_1(\frac{x}{b}).$ 
Define $v := \rho_b u$ then $v\in H^1_0(-B_0,b)$ and satisfies 
\[
-v''-z v\,=\, h\,+\,2(\rho_b'u)'-\rho_b''u.
\]
The right-hand side vanishes on $(-B_0,0)$ so that, in order to use proposition \ref{prop:controlvosc}, 
we have to estimate its $H^{-1}$ norm. The strategy is the same as in the proofs of lemmas 
\ref{lem:normhklm} and \ref{lem:estnormhsm}. 

An integration by parts followed by the use of the Cauchy-Schwarz inequality gives 
\begin{eqnarray*}
\left| \int_{-B_0}^{B_1} (\rho_b'u)'\phi \right| &\leq&  \| \rho' u\|_{L^2(0,b)}\| \phi'\|_{L^2}\\
&\leq & \frac{C}{b}\| u\|_{L^2(0,b)} \|\phi'\|_{L^2}.
\end{eqnarray*}
Thus,
\[
\| (\rho_b' u)'\|_{H^{-1}}\leq \frac{C}{b}\| u\|_{L^2(0,b)}.
\]
The third term can be estimated using the same method.  Indeed,
\begin{eqnarray*}
\left| \int \rho_b'' u \phi\right | & =&\left| \int_0^b \left(\int_0^x \rho_b''(y)u(y)dy\right)\phi'(x)\,dx\right|\\
&\leq & \| \int_0^x \rho_b''(y)u(y)dy\|_{L^2(0,b)}\| \phi'\|_{L^2}.
\end{eqnarray*}
Using again Cauchy-Schwarz inequality and the fact that $|\rho_b''(y) |\leq Cb^{-2}$ we get 
\[
\left|  \int_0^x \rho_b''(y)u(y)dy\right |\leq Cb^{-2}\|u\|_{L^2(0,b)}\sqrt{x}.
\]

We obtain 
\begin{eqnarray*}
\| \int_0^x \rho_b''(y)u(y)dy\|_{L^2(0,b)}&\leq& Cb^{-2}\|u\|_{L^2(0,b)}\| \sqrt{x}\|_{L^2(0,b)}\\
&\leq & Cb^{-1}\|u\|_{L^2(0,b)}.
\end{eqnarray*}
It follows that 
\[
\| h\,+\,2(\rho_b'u)'-\rho_b''u\|_{H^{-1}(-B_0,b)}\leq \|h\|_{H^{-1}(-B_0,b)}\,+\,C b^{-1}\|u\|_{L^2(0,b)}.
\]
We obtain the theorem by plugging this bound into the estimates of the proposition \ref{prop:controlvosc}. 
\end{proof}

\end{appendix}


\begin{thebibliography}{XX}
\bibitem{Ana}
N.~Anantharaman. 
\newblock Entropy and the localization of eigenfunctions.
\newblock {\em Ann. of Math. (2)}, 168(2):435--475, 2008. 

\bibitem{Ana-Non}
N.~Anantharaman and S.~Nonnenmacher. 
\newblock Half-delocalization of eigenfunctions for the Laplacian on an Anosov manifold.
\newblock Festival Yves Colin de Verdi\`ere, {\em Ann. Inst. Fourier (Grenoble)},57(7):2465--252, 2007.

\bibitem{BSS} 
A. B\"acker, R. Schubert, and P. Stifter. 
\newblock {On the number of bouncing ball modes in billiards.}
\newblock {\em J. Phys. A}, {30} :6783--6795, 1997.

\bibitem{BFN}
S.~De Bi\` evre, F.~  Faure, and S.~Nonnenmacher. 
\newblock Scarred eigenstates for quantum cat maps of minimal periods.
\newblock {\em Comm. Math. Phys.}, 239(3):449--492, 2003.

\bibitem{Bu92} N.~Burq, {Control for Schrodinger equations on product manifolds}, Unpublished, 1992.

\bibitem{BHW} 
N. Burq, A. Hassell, and J. Wunsch.
\newblock { Spreading of quasimodes in the Bunimovich stadium}.
\newblock {\em Proceedings of the American Mathematical Society}, 135:1029--1037, 2007.

\bibitem{BZ3}
N.~Burq and M. Zworski.
\newblock Bouncing ball modes and quantum chaos.
\newblock {\em SIAM Review}, 47(5), 43-49, 2005.

\bibitem{CdV}
Y.~Colin de Verdi\`ere. 
\newblock Ergodicit\'e et fonctions propres du laplacien.
\newblock {\em Comm. Math. Phys.}, 102(3):497--502, 1985. 


\bibitem{GerLei}
P.~G\'erard, and \'E.~Leichtnam. 
\newblock Ergodic properties of eigenfunctions for the Dirichlet problem.
\newblock {\em Duke Math. J.}, 71(2):559--607, 1993.  

\bibitem{Hass}
A.~Hassell. 
\newblock Ergodic billiards that are not quantum unique ergodic.
\newblock With an appendix by the author and Luc Hillairet, {\em Ann. of Math. (2)}, 171(1):605-619, 2010.  

\bibitem{HJ}
L. Hillairet and C. Judge.
\newblock Generic spectral simplicity of polygons.
\newblock {\em Proceedings of the American Mathematical Society}, 137:2139--2145, 2009.

\bibitem{Lin}
E.~ Lindenstrauss. 
\newblock Invariant measures and arithmetic quantum unique ergodicity.
\newblock {\em Ann. of Math. (2)}, 163(1):165--219,2006.

\bibitem{M1}
 {J. Marzuola}.
\newblock {Eigenfunctions for partially rectangular billiards},
\newblock {\em Communications in Partial Differential Equations}, {31}, {775-790}, {2007}.

\bibitem{olver}
{F.W.J. Olver}.
\newblock {\it Asymptotics and Special Functions},
\newblock {Academic Press}, New York-London, 1974.

\bibitem{Shn}
A.~Schnirelman.
\newblock Ergodic properties of eigenfunctions.
\newblock {\em Uspehi Mat. Nauk}, 29(6):181--182, 1974. 

\bibitem{Zel}
S.~Zelditch. 
\newblock Uniform distribution of eigenfunctions on compact hyperbolic surfaces.
\newblock {\em Duke Math. J.}, 55(4):919-941, 1987.  

\bibitem{ZZ}
S.~Zelditch, and M.~Zworski. 
\newblock Ergodicity of eigenfunctions for ergodic billiards.
\newblock {\em Comm. Math. Phys.}, 175(3):673--682, 1996.  


\end{thebibliography}
\end{document}